\documentclass{article}
\usepackage{amsmath,amsthm,amssymb,amscd}
\usepackage[all,cmtip,color]{xy}
\usepackage[german,english]{babel}
\usepackage{amsmath}
\usepackage{amssymb}
\usepackage{hyperref}
\usepackage{mathrsfs} 
\usepackage{amssymb}
\usepackage{todonotes}
\usepackage{amssymb}
\usepackage[all,cmtip]{xy}
\usepackage{hyperref}
\usepackage[margin = 2.5 cm]{geometry}
\usepackage{colonequals,enumerate}
\usepackage{comment}
\usepackage{bbm}

\usepackage{authblk}
\usepackage{ dsfont }

\usepackage[export]{adjustbox}
\usepackage{graphicx}

\usepackage{enumitem}

\setlist[enumerate,1]{label={(\roman*)}}
\setlist[enumerate,2]{label={(\alph*)}}

\DeclareMathOperator{\Oc}{\mathcal{O}}
\DeclareMathOperator{\Pc}{\mathcal{P}}

\newcommand{\Gb}{\mathbb{G}}

\DeclareMathOperator{\Hom}{\mathsf{Hom}}

\DeclareMathOperator{\reg}{reg}

\DeclareMathOperator{\A}{{\mathsf{A}}}

\DeclareMathOperator{\Pb}{\mathbb{P}}

\renewcommand{\lim}{\mathsf{lim}}

\DeclareMathOperator{\ani}{ani}

\newcommand\toover[1]{\mathrel{\smash{\overset{#1}{\to}}}}
\newcommand\varto[1]{\mathrel{\hbox to #1pt{\rightarrowfill}}}

\DeclareMathOperator{\id}{id}

\newcommand{\Vv}{\mathsf{V}}

\DeclareMathOperator{\Lie}{Lie}

\DeclareMathOperator{\Zb}{\mathbb{Z}}

\DeclareMathOperator{\Mb}{\mathbb{M}}

\DeclareMathOperator{\Gal}{Gal}
\DeclareMathOperator{\Aut}{\mathsf{Aut}}

\DeclareMathOperator{\PGL}{PGL}
\DeclareMathOperator{\SL}{SL}

\DeclareMathOperator{\Ee}{\mathcal{E}}

\DeclareMathOperator{\Tb}{\mathbb{T}}

\newcommand{\Mc}{\mathcal{M}}

\DeclareMathOperator{\Tr}{\mathsf{Tr}}

\DeclareMathOperator{\Spec}{\mathsf{Spec}}

\DeclareMathOperator{\ad}{\mathsf{ad}}

\DeclareMathOperator{\Qb}{\mathbb{Q}}

\newcommand{\BF}{{\mathbb{F}}}
\newcommand{\BG}{{\mathbb{G}}}
\newcommand{\BH}{{\mathbb{H}}}

\newcommand{\BM}{{\mathbb{M}}}

\newcommand{\BT}{{\mathbb{T}}}

\newcommand{\BZ}{{\mathbb{Z}}}

\newcommand{\CE}{{\mathcal E}}

\DeclareMathOperator{\Pic}{\mathsf{Pic}}

\DeclareMathOperator{\et}{\mathsf{et}}

\DeclareMathOperator{\Out}{Out}

\DeclareMathOperator{\Tw}{\mathsf{Tw}}
\renewcommand{\Out}{\mathsf{Out}}
\DeclareMathOperator{\Sect}{\mathsf{Sect}}

\newcommand{\defeq}{\colonequals}
\let\into\hookrightarrow

\theoremstyle{definition}
\newtheorem{definition}{Definition}[section]
\newtheorem{construction}[definition]{Construction}
\newtheorem{rmk}[definition]{Remark}

\theoremstyle{plain}
\newtheorem{theorem}[definition]{Theorem}
\newtheorem{proposition}[definition]{Proposition}
\newtheorem{corollary}[definition]{Corollary}

\newtheorem{lemma}[definition]{Lemma}

\newtheorem{example}[definition]{Example}

\begin{document}

\title{Twisted Higgs bundles and coendoscopy}

\author[1]{Michael Groechenig\thanks{\url{michael.groechenig@utoronto.ca}}}
\author[2]{Xuanyou Li\thanks{\url{lixuanyo21@mails.tsinghua.edu.cn}}}
\author[3]{Dimitri Wyss\thanks{\url{dimitri.wyss@epfl.ch}}}
\author[4]{Paul Ziegler\thanks{\url{paul.ziegler@mathematik.uni-regensburg.de}}}
\affil[1]{University of Toronto}
\affil[2]{Tsinghua University}
\affil[3]{\'Ecole Polytechnique F\'ed\'erale de Lausanne}
\affil[4]{Universit\"at Regensburg}

\renewcommand\Authands{ and }

\maketitle
\abstract{This short note is devoted to the study of $G$-Higgs bundles twisted by a central gerbe. These objects arise naturally in the decomposition of the inertia stacks of $G$-Higgs bundles in terms of coendoscopic data. We establish that stabilised point-counts and cohomology are insensitive to the central twist. Along the way we show an analogue of Ng\^o's product formula for twisted Hitchin fibres.}

\tableofcontents

\section{Introduction}

Let $k$ be a perfect field and $X$ a smooth, projective, geometrically connected curve of genus $g$ over $k$. For $H \to X$ a smooth group scheme with centre $Z=Z(H)$ and $\beta$ a $Z$-gerbe on $X$ we study the notion of a $\beta$-twisted $H$-Higgs bundle on $X$, similar to the notion of sheaves twisted by a gerbe. 

The moduli stacks $\Mb_H^{\beta}$ of such twisted Higgs bundles share many formal properties with their untwisted counterparts and give in particular a new interpretation for the moduli spaces of $\SL_n$-Higgs bundles appearing in the topological mirror symmetry conjecture of Hausel-Thaddeus \cite{HT03}, see Example \ref{slex}.

Furthermore these twisted moduli stack appear naturally in the description of the cyclotomic inertia stack of usual Higgs bundles. For this let $D$ be a fixed line bundle on $X$ of even degree and $G$ a quasi-split reductive group scheme $G\to X$. We write $\Mb_G = \Mb_G(X,D)$ for the moduli stack of $G$-Higgs bundles on $X$ with values in $D$. The open substack $\Mb_G^{ani} \subset \Mb_G$ of anisotropic Higgs bundles is a smooth Deligne-Mumford stack and a central object in Ngô's approach to the Langlands-Shelstad fundamental lemma via the geometric stabilization theorem \cite[Théorème 6.4.1]{Ng10}. 

In the $p$-adic approach to geometric stabilization \cite{GWZ20b}, an important ingredient is the description of the twisted inertia stack $I_{\hat \mu} \Mb_G^{ani} = \Hom(B\hat{\mu}, \Mb_G^{ani})$ which roughly speaking claims an equivalence of groupoids
\begin{equation}\label{weq} I_{\hat \mu} \Mb_G^{ani}(k) \cong \bigsqcup_{\CE} \Mb_{H_\CE}^{ani}(k). \end{equation}

Here $\CE$ ranges over coendoscopic data, see below, and $\Mb_{H_\CE}$ parametrizes Higgs bundles for the coendoscopic group $H_\CE$ associated with $\CE$. Unfortunately, as noticed by the second author, the equivalence \eqref{weq} is incorrect as stated. However, as we will explain below, this is mistake does not negatively affect the overall strategy of the proof of \cite{GWZ20b}.

To obtain a correct version of \eqref{weq}, one has to replace $\Mb_{H_\CE}$ by the moduli stack $\Mb_{\CE}=\Mb_{H_{\CE}^{\beta_{\CE}}}$ of $H_\CE$-Higgs bundles twisted by a certain $Z(H_\CE)$-gerbe.

\begin{theorem}[Theorem \ref{dec-IM}] Over  $\overline{k}$ there is an equivalence of stacks
	\begin{equation*}
		  \bigsqcup_{\mathcal{E}} \widetilde{\mathbb{M}}^{G-\infty,\mathrm{ani}}_{\mathcal{E}} \to I_{\hat\mu}\widetilde{\mathbb{M}}^{\mathrm{ani}}_{G},
	\end{equation*}
	where the disjoint union is taken over isomorphism classes of coendoscopic data $\CE$. 
\end{theorem}
For a version over $k$ see Remark \ref{desc}. In order for the arguments in \cite[Section 6]{GWZ20b} to remain valid, we need to relate the count of rational points  $\Mb_{H_\CE}(k)$ and $\Mb_{\CE}(k)$ when $k=\BF_q$ is a finite field of large enough residue characteristic. Concretely we consider the Hitchin fibrations
\[\widetilde{\Mb}^{ani}_{H_\CE} \rightarrow \widetilde{\A}^{ani}_{H_\CE} \leftarrow \widetilde{\Mb}^{ani}_{\CE}.   \]

\begin{theorem}[Theorem \ref{stabct}] For every $a \in \widetilde{\A}^{ani}_{H_\CE}(k)$ the Hitchin fibers $\widetilde{\Mb}_{H_\CE,a}$ and  $\widetilde{\Mb}_{\CE,a}$ over $a$ satisfy
\[\#^{stab}\widetilde{\Mb}_{H_\CE,a}(k) = \#^{stab}\widetilde{\Mb}_{\CE,a}(k),    \]
where $\#^{stab}$ denotes the stabilized point count.
\end{theorem}

In Section \ref{stpc} we give two proofs for this theorem, one using $p$-adic integration and one by establishing an analogue of Ngô's product formula \cite[Proposition 4.15.1]{Ng10}.

The latter is of independent interest as it highlights the striking similarity of the local geometry of the moduli spaces $\BM_{G}^{\beta}$ for varying choices of the central gerbe $\beta$.

\begin{proposition}[Product formula, Proposition \ref{prop:productformula}]
    There is an equivalence between categories of $k$-points
		\[[\mathbb{M}_{H,a}^\beta/\mathbb{P}_{H,a}](k)\cong
		\prod_{\nu\in |X-U|}[\mathbb{M}_{H,a,\nu}^\beta/\mathbb{P}_{H, a,\nu}](k).\]
\end{proposition}

\begin{rmk}
Soon after having completed a first draft of the present paper, the authors learned that the moduli problem $\BM_{G}^{\beta}$ appears in the recent work of Bu--Davison--Ib\'añez N\'uñez--Kinjo--P\v adurariu \cite[Section 10.3]{BDINKP}, where it is used to formulate a generalisation of the Hausel--Thaddeus conjecture \cite{HT03} for arbitrary semisimple structure groups.
\end{rmk}

\subsubsection*{Acknowledgments}X. L. thanks Yihang Zhu for continuous support and helpful discussions.

M. G. was supported by an NSERC discovery grant and an Alfred P. Sloan fellowship. D.W. was supported by the Swiss National Science Foundation [no. 218340]. P.Z. was funded by the Deutsche Forschungsgemeinschaft (DFG, German Research Foundation) – Projektnummer 547483711.

\section{Twisted $H$-bundles}


\subsection{Definitions}

Let $U$ be a scheme, $H \to U$ be a smooth group $U$-scheme with centre $Z=Z(H) \to U$. We denote by $\beta$ a $Z$-gerbe and by $[\beta] \in H^2_{\text{\'et}}(U,Z)$ its classifying class in \'etale cohomology. In the following we define the notion of a $\beta$-twisted $H$-bundle, a variation on the concept of a twisted sheaf. We begin with a conceptual treatment and subsequently turn to a concrete description in terms of cocycles just after the following definition.

\begin{definition}\label{defi:twisted-bundle}
A $\beta$-twisted $H$-bundle $\Vv$ is given by a section represented by a dashed arrow in the following Cartesian diagram of (higher) $U$-stacks:
\[
\xymatrix{
\Tw_H^{\beta} \ar[r] \ar[d] & U \ar@{-->}@/_1.0pc/[l]_{\Vv} \ar[d]_{\beta} \\
B_U(H/Z) \ar[r] & B_U^2Z.
}
\]
\end{definition}

In other words, a $\beta$-twisted $H$-bundle over $U$ is given by a $H/Z$-torsor $V$ over $U$ together with an isomorphism $[V/H]\cong \beta$ of $Z$-gerbes.

It is also possible to give a cocycle description of $\beta$-twisted $H$-bundles. For this purpose we represent the class $[\beta] \in H^2_{\text{\'et}}(U,Z(H))$ by an \'etale $3$-cocycle $(\beta_{ijk}) \in Z(U_{ijk})$, where $(U_i)_{i \in I}$ denotes the corresponding \'etale cover of $U$ with triple intersections $U_i \times_U U_j \times_{U} U_k$ denoted by $U_{ijk}$. It is then possible to represent $\mathcal{V}$ by sections $h_{ij} \in H(U_{ij})$ satisfying the \emph{twisted cocycle relation}
\[h_{ij}h_{jk} = \beta_{ijk} \cdot{} h_{ik}.\]
Intuitively speaking, a $\beta$-twisted $H$-bundle $\Vv$ is therefore \'etale locally given by a $H$-bundles $\Vv_i$ defined on \'etale opens $U_i \to U$ together with glueing data that only satisfies a twisted version of the usual glueing condition for $H$-bundles.

From this point of view, it is clear that there is also a related notion of $\beta$-twisted $H$-Higgs bundles. We now spell out a definition independent of a chosen \'etale cover. For this purpose, we fix a line bundle $D$ on $U$ and define $\beta$-twisted $H$-Higgs bundles with coefficients in $D$. To shorten the notation we will denote the quotient $H/Z$ by $\overline{H}$.

\begin{lemma} \label{TwPushout}
    There is a canonical isomorphism
    \begin{equation*}
        \Tw_H^\beta \cong B_U H \times^{B_U Z} \beta.
    \end{equation*}
\end{lemma}
\begin{proof}
    The two projections $B_U H \times^{B_U Z} \beta \to B_U H/ B_U Z=B_U \bar H$ and $B_U H \times^{B_U Z} \beta \to \beta/B_U Z=U$ induce a morphism $B_U H \times^{B_U Z} \beta \to \Tw_H^\beta$ of $B_U Z$-torsors over $U$. Such a morphism is automatically an isomorphism.
\end{proof}
\begin{definition}
Consider the following Cartesian diagram of (higher) $U$-stacks
\[
\xymatrix{
& [h_D/H]^{\beta} \ar[ld] \ar[r] \ar[d] &\Tw_H^{\beta} \ar[r] \ar[d] & U \ar@{-->}@/_1.0pc/[l]^{\Vv} \ar@{-->}@/_2.0pc/[ll]_{(\Vv,\theta)} \ar[d]_{\beta} \\
(\mathfrak{c}_h)_D & \ar[l] [h_D/\overline{H}] \ar[r] & B_U(\overline{H}) \ar[r] & B_U^2Z.
}
\]
A \emph{$\beta$-twisted $H$-Higgs bundle $(\Vv,\theta)$ with coefficients in $D$} is given by a section $U \to [h_D/H]^{\beta}$. The morphism given by the composition $U \to (\mathfrak{c}_h)_D$ yields an element in $H^0(U,(\mathfrak{c}_h)_D)$, which we will denote by $\chi(\theta)$.
\end{definition}

\begin{rmk} Since we are only interested in the case where $U$ is a (relative) curve, we omit the integrability condition one usually imposes on Higgs bundles on a higher dimensional base. 
\end{rmk}

\subsection{Relation to the classical case}

Henceforth we will study $\beta$-twisted $H$-bundles on a curve $X$ over a perfect field $k$ with coefficients in a line bundle $D$. If the characteristic of $k$ is positive, we assume it is bigger than twice the coxeter number of $H$.

\begin{definition}
We denote by $\Mb_H^{\beta}$ the moduli stack of $\beta$-twisted $H$-Higgs bundles. The $Z$-rigidification of this stack will be denoted by $\Mc_H^{\beta}$. The association $(\Vv,\theta) \mapsto \chi(\theta)$ defines a Hitchin morphism $\chi\colon \Mc_{H}^{\beta} \to \A_H$.
\end{definition}

First of all, we observe that the moduli stack $\Mb^{\beta}_{H}$ is closely related to a classical moduli problem of Higgs bundles.

\begin{construction}\label{const:untwist}
The map $H \to \overline{H} = H/Z$ gives rise to a morphism of algebraic stacks $\Mb_H^{\beta} \to \Mb_{\overline{H}}$ that fits into a commutative diagram
\[
\xymatrix{
\Mb_H^{\beta} \ar[r] \ar[d] & \Mb_{\overline{H}} \ar[d] \\
\A_H \ar[r] & \A_{\overline{H}}.
}
\]
\end{construction}

\begin{proposition}\label{torus-field-trivial}
	Let $K/\bar k$ be an extension of transcendental degree $1$ and $T$ be a torus over $K$. Then $H^i(K, T)=0$ for $i\geq 1$.
\end{proposition}

\begin{proof}
	By Tsen's theorem, the conditions in \cite[p.169 Proposition 11]{MR0354618} are satisfied. 
	By \cite[p.170 Application]{MR0354618}, for any finite Galois extension $L/K$ on which $T_L$ is split, we have that  $T(L)$ is cohomologically trivial, namely, $\hat H^n(\Gal(L/K), T(L))=0$ for any $n$. In particular, $H^n(\Gal(L/K), T(L))=0$ for any $n\geq 1$. Varying $L$ we conclude the claim.
\end{proof}

\begin{lemma}\label{torus-curve-trivial}
	Let $T$ be a group scheme on $X$ such that its restriction to the generic point it is a torus. Then $H^i_{\et}(X_{\bar k},T)$ is trivial for $i\geq 2$. 
\end{lemma}

\begin{proof}
	Proposition \ref{torus-field-trivial} and the standard arguments in \cite[IX 4]{SGA43} imply the claim. More precisely, let $\eta$ be the generic point of $X_{\bar k}$, $j:\eta\to X_{\bar k}$ be the natural map, $T\to Rj_*j^*T$ be the canonical map, and 
	\begin{equation}
		T\to Rj_*j^*T\to \Delta\to \label{torus-field-trivial:seq}
	\end{equation}
	be a distinguished triangle. By \cite[IX 4.1]{SGA43} each $\mathcal H^i (\Delta)$ is a skyscraper sheaf. For any étale map $U\to X_{\bar k}$, by Proposition \ref{torus-field-trivial} we have $H^i_{\et}(\eta\times_{ X_{\bar k}} U, j^*T)=0$ for $i\geq 1$. This implies that $R^ij_*j^*T$ vanishes for $i\geq 1$. Thus for $i\geq 1$, $\mathcal H^i (\Delta)=0$, hence $H^i_{\et}(X,\Delta)=0$. From the long exact sequence associated to (\ref{torus-field-trivial:seq}) we conclude that $H^i_{\et}(X_{\bar k},T)=0$ for $i\geq 2$. 
\end{proof}

\begin{lemma}\label{vanishing-H2}
    Let $(A,m)$ be an excellent strict Henselian local ring, let $X\to \Spec A$ be a proper smooth curve and let $T$ be a torus on $X$. Then $H^2_{\et}(X,T)$ is trivial. 
\end{lemma}

\begin{proof}
    By \cite[III. Theorem 2.17]{MEC}, we may identify $H^2_{\et}(X,T)$ with the \v Cech cohomology group $\check{H}^2_{\et}(X,T)$. Thus it suffices to show that any $T$-gerbe on $X$ is trivial.

    First, assume that $A$ is complete. Let $Y$ be a $T$-gerbe over $X$. 
    For $n\geq 0$, define $A_n=A/m^{n+1}$, and let $X_n=X\times_A A_n$, $Y_n=Y\times_A A_n$.
    Note that $Y/A$ has an affine diagonal. By \cite[Theorem 7.4]{bhatt2017tannaka} we have an equivalence of categories 
    \[\mathrm{Map}_A(X,Y)\cong \varprojlim \mathrm{Map}_{A_n}(X_n,Y_n), \quad \mathrm{Map}_A(X,X)\cong \varprojlim \mathrm{Map}_{A_n}(X_n,X_n) .\]
    We conclude that \[\mathrm{Map}_X(X,Y)\cong \varprojlim \mathrm{Map}_{X_n}(X_n,Y_n).\]
    For each $n$, there is an exact sequence of \'etale sheaves on $X_{n}$
    \[0\to m^n/m^{n+1}\otimes_A  \Lie_{X} T\to  T_{n}\to T_{n-1}\to 0.\]
    Since $\Lie_{X} T$ is a coherent sheaf, $H^i_{\et}(X_n, m^n/m^{n+1}\otimes_A  \Lie_{X} T)\cong H^i_{\mathrm{Zar}}(X_n, m^n/m^{n+1}\otimes_A  \Lie_{X} T)$ vanishes for $i\geq 1$. Combining this with Lemma \ref{torus-curve-trivial}, which implies $H^2_{\et}(X_0,T_0)=0$, we conclude that 
    the natural maps $H^1_{\et}(X_{n+1},T_{n+1})\to H^1_{\et}(X_n,T_n)$ are surjective, 
    and each $H^2_{\et}(X_n,T_n)$ vanishes. In particular, $\mathrm{Map}_{X_n}(X_n,Y_n)$ is non-empty and $\mathrm{Map}_{X_{n+1}}(X_{n+1},Y_{n+1})\to \mathrm{Map}_{X_n}(X_n,Y_n)$ is surjective. Therefore $\mathrm{Map}_X(X,Y)$ is non-empty, which implies that $Y$ is trivial.

    In general, by Popescu’s smoothing theorem \cite[\href{https://stacks.math.columbia.edu/tag/07GC}{Tag 07GC}]{stacks-project}, there is a filtered direct system $\{A_j\}_{j\in J}$ of smooth $A$-algebras such that $\widehat{A}\cong \varinjlim_{j\in J} A_j$. Let $Y$ be a $T$-gerbe on $X$, by the previous paragraph $Y_{\widehat{A}}$ is trivial, so we may find $j$ such that $Y_{A_j}$ is trivial. 
    Since $A$ is strict Henselian, we may find a section $A_j\to A$. Pulling back along this section we conclude that $Y$ is trivial.
\end{proof}

\begin{corollary}\label{isomorphism-H2}
   Let $M$ be a smooth Deligne-Mumford stack over $k$, and let $p$ be the projection $X\times_k M\to M$. Then $R^2p_{*,\et}Z^0$ is trivial.  
   Moreover, $Z\to \pi_0(Z)$ induces an isomorphism
   \[R^2p_{*,\et}Z\cong R^2p_{*,\et}\pi_0(Z)\cong H^2_{\et}(X, \pi_0(Z)).\]
   Here we regard $H^2_{\et}(X, Z)\cong H^2_{\et}(X, \pi_0(Z))$ as a constant sheaf on $M$.
\end{corollary}
\begin{proof}
    We may assume $M$ is a smooth scheme over $k$. Then the first claim follows from Lemma \ref{vanishing-H2}. As a corollary, $Z\to \pi_0(Z)$ induces an injection
   \[R^2p_{*,\et}Z\hookrightarrow R^2p_{*,\et}\pi_0(Z)\cong H^2_{\et}(X, \pi_0(Z)).\]
   Here $\pi_0(Z)$ is a torsion sheaf, so by the proper base change theorem $R^2p_{*,\et}\pi_0(Z)\cong H^2_{\et}(X, \pi_0(Z))$. 
   Note that, by Lemma \ref{torus-curve-trivial}, we have $H^2_{\et}(X, Z)\cong H^2_{\et}(X, \pi_0(Z))$. We have a Cartesian diagram
    \[
\xymatrix{
X\times_k M\ \ar[r]_q \ar[d]_p &  X \ar[d]^{p'} \\
M \ar[r]_{q'} & \Spec k,
}
\]
This induces the base change map $$f\colon H^2_{\et}(X, Z)\cong q'^*Rp'_{*,\et}Z_X\to Rp_{*,\et}q^*Z_X\to Rp_{*,\et}Z_{X\times_k M}.$$ 
By the functoriality of the base change map, we have the following commutative diagram
\[
\xymatrix{
{H^2_{\et}(X, Z)} \ar[r]^-\sim \ar[d]_f & {H^2_{\et}(X, \pi_0(Z))} \ar[d]^\sim \\
R^2p_{*,\et}Z \ar[r]                           & R^2p_{*,\et}\pi_0(Z).
}
\]
This implies that the composition
\[H^2_{\et}(X, Z)\xrightarrow{f} R^2p_{*,\et}Z\hookrightarrow  R^2p_{*,\et}\pi_0(Z)\cong H^2_{\et}(X, \pi_0(Z))\cong H^2_{\et}(X, Z)\]
is the identity, so $R^2p_{*,\et}Z\hookrightarrow  H^2_{\et}(X, Z)$ is an isomorphism.
\end{proof}

\begin{construction}
    Let $(E,\theta)$ be the universal $\overline{H}$-Higgs bundle on $X\times_k \Mb_{\overline{H}}$. We may associate it to the $Z$-gerbe $[E/H]$, which defines a global section $\beta^{\mathrm{univ}}\in \Gamma(\Mb_{\overline{H}}, R^2p_{*,\et}Z)\cong \Gamma(\Mb_{\overline{H}}, H^2_{\et}(X, \pi_0(Z)))$. For any $\beta\in H^2_{\et}(X, \pi_0(Z))$, let $\Mb_{\overline{H}}^\beta$ be the largest open subset of $\Mb_{\overline{H}}$ on which $\beta=\beta^{\mathrm{univ}}$. We have $\Mb_{\overline{H}}=\bigsqcup_{\beta\in H^2_{\et}(X, \pi_0(Z))}\Mb_{\overline{H}}^\beta$.
\end{construction}

\begin{lemma}\label{lemma:sect-torsor}
The morphism $F:\Mb_{H}^\beta \to \Mb_{\overline{H}}^\beta\times_{{\A}_{\overline{H}}}{\A}_{H}$ is a $\Sect_X(BZ(H))$-torsor.
\end{lemma}

\begin{proof}
For any $k$-scheme $S$, a section in $(\Mb_{\overline{H}}^\beta\times_{{\A}_{\overline{H}}}{\A}_{H})(S)$ can be identified with a $\overline{H}$-torsor $E$ and a section $\theta\in \Gamma(X\times_k S, E\times^{\overline{H}} \mathfrak{h}_D)$. 
This section $(E, \theta)$ lies in the image of $\Mb_{H}^\beta(S)$ if and only if the associated $Z$-gerbe $[E/H]$ on $X\times S$ is isomorphic to $\beta$. Thus, \'etale-locally on $S$ we may find a lifting of $(E,\theta)$ if and only if $\beta=\beta^{\mathrm{univ}}$ holds on $S$. 
The claim then follows from the observation that the groupoid of lifts of a given map $S \to \Mb_{\overline{H}}$ is a torsor under the commutative group stack $\Sect_X(BZ(H))$.
\end{proof}

Recall from \cite[Séction 4.5] {Ng10} the definition of the Prym $\Pb_H$ over $\A_H$ together with its canonical action on $\Mb_H$ relative to $\A_H$, and analogously for $\bar H$. The action of $\Pb_{\overline{H}}$ on $\Mb_{\overline{H}}$ permutes the decomposition $\Mb_{\overline{H}}=\bigsqcup_{\beta\in H^2_{\et}(X, \pi_0(Z))}\Mb_{\overline{H}}^\beta$ as follows:

For any section $a \in \A_H(k)$ with image $\bar a$ in $A_{\overline{H}}(k)$, the associated regular centralizer schemes on $X$ fit into the following exact sequence:
\begin{equation*}
    1 \to Z(H)_{X} \to J_{H,a} \to J_{\overline{H},\bar a} \to 1
\end{equation*}
Hence to any section $p \in \Pb_{\overline{H},\bar a}(k)$ we may associate a $Z(H)_{X}$-gerbe $\delta(p)$ over $X$ measuring the obstruction to lifting $p$ to $\Pb_{H,a}(k)$. From the definition of the action of $\Pb_{\overline H}$ on $\Mb_{\overline H}$ one sees:
\begin{lemma}
    For any $a \in \A_H(k)$ and $\beta\in H^2_{\et}(X, \pi_0(Z))$, any $p \in \Pb_{\overline{H},\bar a}(k)$ sends $\Mb_{\overline{H},\bar a}^\beta$ to $\Mb_{\overline{H},\bar a}^{\beta+\delta(p)}$. 
\end{lemma}

Recall from \cite[Séction 4.5] {Ng10} the open locus $\A^\heartsuit_H \subset \A_H$ corresponding roughly speaking to reduced cameral curves and inside the anisotropic locus \cite[Séction 6.1]{Ng10} $\A^{ani}_H\subset \A^\heartsuit_H$.  

For any $a \in \A^\heartsuit_H(k)$, the regular centralizer $J_a$ on $X_{\bar k}$ is generically a torus, which implies that $H^2_{\et}(X_{\bar k},J_a)$ is trivial by Lemma \ref{torus-curve-trivial}. Hence for such an $a$, the above implies that $\Pb_{\overline H,\bar a}$ permutes the subspaces $\Mb_{\overline{H},\bar a}^\beta$ transitively.

\begin{definition}
    We introduce the following notation:
    \begin{enumerate}
        \item $\Mb_H^{\beta,\heartsuit} \defeq \Mb_H^{\beta} \times_{\A_H} \A_H^\heartsuit $
        \item $\Mb_H^{\beta,\ani} \defeq \Mb_H^{\beta} \times_{\A_H} \A_H^{\ani}$
        \item $\Mb_H^{\beta,\reg} \subset \Mb_H^{\beta}$ is the open substack whose sections are those which factor through the open substack $[h_D/H]^{\beta,\reg} \subset [h_D/H]^{\beta}$ defined by the Cartesian diagram
        \begin{equation*}
            \xymatrix{
[h_D/H]^{\beta,reg} \ar[r] \ar[d] & X  \ar[d]_{\beta} \\
 [h^{\reg}_D/\overline{H}] \ar[r] &  B_X^2Z,         
            }
                    \end{equation*}

            where $h^{\reg}_D \subset h_D$ is the regular locus.
            \item For any of these objects $\Mb_H^{\beta,?}$, we denote the associated $Z$-rigidification by $\Mc_H^{\beta,?}$

    \end{enumerate}
\end{definition}
\begin{proposition}\label{heart}
    \begin{enumerate}
        \item The stack $\Mb_H^{\beta}$ is algebraic.
        \item The stack $\Mb_H^{\beta,\ani}$ is a smooth DM-stack which is proper over ${\A}_{H}^{\ani}$. 

        \item If $\deg(D) \geq 2g$ is even, then the stack $\Mb_{H,a}^{\beta,\reg}$ is dense in the Hitchin fibre for every $a \in \A^{\heartsuit}$. 

        \item The stack $\Mb_H^{\beta,\reg}$ is a $\Pb_{H}$-torsor over $\A^{\heartsuit}_{H}$. 
    \end{enumerate}
\end{proposition}

\begin{proof}
    Claims (i) - (iii) follow directly from Lemma \ref{lemma:sect-torsor} and the corresponding statements for $\Mb_{\bar H}$. 
  
    For (v), since $[\mathfrak{g}^{\reg}/G]^\beta\to \mathfrak{c}$ is a $J$-gerbe, we have an isomorphism 
    $$(\mathrm{act}, \id): BJ\times_{\mathfrak{c}} [\mathfrak{g}^{\reg}/G]^\beta\to [\mathfrak{g}^{\reg}/G]^\beta\times_{\mathfrak{c}} [\mathfrak{g}^{\reg}/G]^\beta.$$
    By taking mapping stacks, we get an isomorphism
    \[(\mathrm{act}, \id): \Pb_{H}\times_{\A^{\heartsuit}_{H}}\Mb_H^{\beta,\reg}\to \Mb_H^{\beta,\reg} \times_{\A^{\heartsuit}_{H}}\Mb_H^{\beta,\reg}.\]
   It remains to show that $\BM_H^{\beta,\reg} \to \A^\heartsuit_H$ is an fppf covering. Lemma \ref{lemma:sect-torsor} implies that this morphism is smooth. Hence it suffices to show that each fiber of this morphism is non-empty. Using Lemma \ref{lemma:sect-torsor} this reduces to showing that every fiber of $\Mb_{\bar H}^{\beta,\reg} \to \A_{\overline H}$ is non-empty. Since, as we have noted above, for each $a \in \A^\heartsuit(\bar k)$ the $\Mb_{\bar H,\bar a}^{\beta,\reg}$ are permuted transitively by $\Pb_{\overline{H},\bar a}$, this follows from the fact that $\Mb_{\overline H}^{\reg}$ is a $\Pb_{\overline H}$-torsor over $\A^\heartsuit$.

\end{proof}

\begin{lemma}
The morphism 
\[
\bigsqcup_{\beta} \Mb^{\beta}_H \to \Mb_{H/Z},
\]
where $\beta$ ranges over a family of representatives for $H^2_{\text{\'et}}(X,Z)$, is surjective.
\end{lemma}
\begin{proof}
This follows from the observation that $H^2_{\text{\'et}}(X,Z)$ measures the obstruction to lifting an $H/Z$-torsor to an $H$-torsor. If the obstruction is non-trivial, a lift to a twisted $H$-torsor exists. 
\end{proof}

\begin{example}\label{slex}
 The moduli space $\Mb_{\PGL_n}$ has several connected components. One has $\pi_0(\Mc_{\PGL_n}) = \Zb/n\Zb = H^2_{\text{fl}}(X_{\bar{k}},\mu_n)$. The component labeled by $0$ corresponds to the image of $\Mc_{\SL_n}$. The other components are isomorphic to the images of moduli spaces $\Mc_{\SL_n}^L$ of ``$\SL_n$"-bundles as defined by Hausel--Thaddeus. Here $\Mc_{\SL_n}^L$ denotes the rigidified moduli space of Higgs bundles with a fixed determinant $L$, where $L$ denotes a fixed line bundle of non-zero degree. 
 
 These components can be seen as the images of moduli spaces of twisted Higgs bundles $\Mb^\beta_{\SL_n}$ in our sense as follows: Since $H^2_{\text{\'et}}(X_{\bar{k}},\Gb_m)=0$ by Lemma \ref{vanishing-H2}, any class $\beta \in H^2(X,Z(\SL_n))=H^2(X,\mu_n)$ comes from a line bundle $L \in H^1(X,\Gb_m).$ Using a choice of auxiliary line bundle $N$ on $X$ of degree one, one can lift the canonical $\Gamma \defeq \Pic^0(X)[n]$-action on $\Mc_{\SL_n}^L$ to an action on $\Mb_{\SL_n}^L$, c.f. \cite[Subsection 2.4]{maneval2025tms}. This gives the following diagram:
 \begin{equation*}
     \xymatrix{
     \Mb_{\SL_n}^L  \ar[r] \ar[d] & [\Mb_{\SL_n}^L/\Gamma]=\Mb^\beta_{\SL_n} \ar[d]\\
     \Mc_{\SL_n}^L \ar[r] & [\Mc_{\SL_n}^L/\Gamma]=\Mb_{\PGL_n}^\beta
     }
 \end{equation*}
\end{example}

\subsection{Moduli spaces associated to coendoscopic data}
We now put ourselves in the situation of \cite[Section 5]{GWZ20b}. In this subsection, we fix a coendoscopic datum $\CE=(\kappa,\rho_\kappa,\rho_\kappa \to \rho)$ for $G$ over $k$. 

\begin{definition}
  Recall from \cite[Construction 5.2]{GWZ20b} that we have a commutative diagram
    \[\xymatrix{
    1\ar[r] &Z(\BH_\kappa) \ar[r]\ar[d] &\tilde{\pi}_0(\kappa) \ar[r]\ar[d] &\pi_0(\kappa)\ar[d] \ar[r] &1\\
    1\ar[r] &\BH_\kappa \ar[r] &(\BG\rtimes \mathrm{Out}(\BG))_\kappa \ar[r] &\pi_0(\kappa) \ar[r] &1,
    }   
    \]
    where $\tilde{\pi}_0(\kappa) \subset (\BG\rtimes \Out(\BG))_\kappa$ is the stabilizer of the pinning of $\BH_\kappa$.

We let $\beta_\CE$ be the $Z_\CE=Z(H_\CE)$-gerbe $[\rho_\kappa/\tilde\pi_0(\kappa)]$ and denote the associated spaces $\Mb_{H_\CE}^{\beta_\CE,?}$ (resp. $\Mc_{H_\CE}^{\beta_\CE,?}$) simply by $\Mb_{\CE}^{?}$ (resp. $\Mc_{\CE}^?$).  
\end{definition}

\begin{proposition} \label{TwistedDesc}
    There are canonical isomorphisms as follows:
    \begin{enumerate}
        \item     $\Tw_{H_\CE}^{\beta_\CE} \cong [\rho_\kappa/ (\BG \rtimes \Out(\BG))_\kappa]$        
    \item    
       $ [\mathfrak{h}_{ \CE,D}/H_\CE]^{\beta_\CE} \cong [\mathbf{h}_{D} \times \rho_\kappa  / (\mathbb{G}\rtimes \mathrm{Out}(\mathbb G))_\kappa].$
    \end{enumerate}

\end{proposition}
\begin{proof}
    (i) We note that by \cite[Construction 5.1]{GWZ20b}, there is a canonical isomorphism $$BH_\CE=[\rho_\kappa/ (\BH_\kappa \rtimes \pi_0(\kappa))],$$ using which $[\rho_\kappa/(\BG \rtimes \Out(\BG))_\kappa]$ naturally becomes a $BH_\CE$-torsor over $X=[\rho_\kappa/\pi_0(\kappa)]$.
    
    Similarly, the canonical morphism
    \begin{equation*}
        \beta=[\rho_\kappa/\tilde{\pi}_0(\kappa)] \to [\rho_\kappa/(\BG \rtimes \Out(\BG))_\kappa]
    \end{equation*}
    is $B Z_\CE=[\rho_\kappa/(Z(\BH_\kappa) \rtimes \pi_0(\kappa))]$-equivariant. So this gives a canonical identification $$[\rho_\kappa/(\BG \rtimes \Out(\BG))_\kappa] \cong BH_\CE \times^{B Z_\CE} \beta.$$ Using Lemma \ref{TwPushout} this gives the first isomorphism. 

    (ii) Consider the following diagram of groups with exact rows:
\[
\xymatrix{
1 \ar[r] & \BH_{\kappa} \ar[r] \ar[d] & (\Gb \rtimes \Out(\Gb))_{\kappa} \ar[r] \ar[d] & \pi_0(\kappa) \ar[d] \ar[r] & 1 \\
1 \ar[r] & \BH_{\kappa}/Z(\BH_{\kappa}) \ar[r] & \ar@/_1.0pc/@{-->}[l](\Gb \rtimes \Out(\Gb))_{\kappa} / Z(\BH_{\kappa}) \ar[r] &  \pi_0(\kappa) \ar[r] & 1
}
\]
Here the bottom exact sequence admits a canonical splitting as follows: If we conjugate the given pinning on $\BH_\kappa$ by any section of $(\BG \rtimes \Out(\BG))_\kappa$, we obtain a new pinning of $\BH_\kappa$, which differs from the original one by a unique section $\BH_\kappa/Z(\BH_\kappa)$. This gives the dashed arrow splitting the sequence. So we obtain a canonical isomorphism
\begin{equation} \label{SplitIso}
    (\Gb \rtimes \Out(\Gb))_{\kappa} / Z(\BH_{\kappa}) \cong \BH_{\kappa}/Z(\BH_{\kappa}) \rtimes \pi_0(\kappa)
\end{equation}

Using \eqref{SplitIso} and (i), we obtain the following Cartesian diagram
\begin{equation*}
    \xymatrix{
        [\mathbf{h}_{D} \times \rho_\kappa  / (\mathbb{G}\rtimes \mathrm{Out}(\mathbb G))_\kappa] \ar[r] \ar[d] & [\rho_\kappa  / (\mathbb{G}\rtimes \mathrm{Out}(\mathbb G))_\kappa]=\Tw_{H_\CE}^{\beta_\CE} \ar[d] \\
        [\mathbf{h}_{D} \times \rho_\kappa  / ((\mathbb{G}\rtimes \mathrm{Out}(\mathbb G))_\kappa/Z(\BH_\kappa))]\ar@{=}[d] \ar[r] & [\rho_\kappa  / ((\mathbb{G}\rtimes \mathrm{Out}(\mathbb G))_\kappa/Z(\BH_\kappa))] \ar@{=}[d]\\
        [\mathbf{h}_{D} \times \rho_\kappa  / (\BH_{\kappa}/Z(\BH_{\kappa}) \rtimes \pi_0(\kappa))] \ar@{=}[d] \ar[r] & [\rho_\kappa  / (\BH_{\kappa}/Z(\BH_{\kappa}) \rtimes \pi_0(\kappa))]\ar@{=}[d]\\
        [h_D/\overline{H}_\CE] \ar[r] & B \overline{H}_\CE
    }
\end{equation*}
\end{proof}

    \begin{definition}
       Using the objects defined in \cite[Construction 5.9]{GWZ20b}, for any of the objects $\Mb_\CE^*$ defined above, we let $$\widetilde\Mb_\CE^* \defeq \Mb_\CE^* \times_{\A_\CE} \widetilde \A_\CE.$$ 
    \end{definition}
\begin{construction}
	Using Proposition \ref{TwistedDesc} we obtain the natural map of groupoids 
	\[[\mathfrak{h}_{ \CE,D}/H_\CE]^{\beta_\CE}\cong [\mathbf{h}_{D} \times \rho_\kappa  / (\mathbb{G}\rtimes \mathrm{Out}(\mathbb G))_\kappa]\to [\mathbf{g}_{D} \times \rho_\kappa  / \mathbb{G}\rtimes \mathrm{Out}(\mathbb G)]\cong [\mathfrak g_D/G],\]
	which induces a map $\mu_\CE\colon \mathbb{M}_{\mathcal{E}}\to \mathbb{M}_{G}$. Moreover, since $\kappa:\hat{\mu}\to \mathbb{T}$ maps into the center of $(\mathbb{G}\rtimes \mathrm{Out}(\mathbb G))_\kappa$ by definition, every twisted Higgs bundle $F\in \mathbb{M}_{\mathcal{E}}$ is equipped with a morphism $\hat{\mu}\toover{\kappa} \Aut (F)$. Using the functoriality of twisted inertia stacks we obtain a morphism of stacks
	\[\mu_{\mathcal E}: \mathbb{M}_{\mathcal{E}}\to I_{\hat\mu}\mathbb{M}_{\mathcal{E}}\to I_{\hat\mu}\mathbb{M}_{G}.\]

    By passing to coarse moduli spaces we obtain a morphism
    \begin{equation*}
        \nu_\CE\colon \mathfrak{c}_\CE = \mathfrak{h}_{\CE,D}/H_\CE \to \mathfrak{c}=\mathfrak{g}_D/G,
    \end{equation*}
    which induces a morphism $\A_\CE \to \A$. One can check that this morphism coincides with the one defined in \cite[1.9]{Ng10}. For this one may pass to a finite extension to $k$ over which $\rho_\kappa$ has a section, which then trivializes $\beta$ and identifies $\mu_\CE$ with the canonical morphism $[\mathbf{h}_D/\mathbb{H}] \to \mathbf{g}/\Gb$.
    
    Finally by pulling back $\mu_\CE$ to the \'etale-open $\widetilde\A^{\infty,\ani} \to \A$ from \cite[Construction 5.9]{GWZ20b} we obtain the following morphism $\tilde\mu_\CE$:
\begin{equation*}
    \xymatrix{
        \widetilde\Mb^{G-\infty,\ani}_\CE \ar[r]^{\tilde\mu_\CE} \ar[d] & I_{\hat\mu} \widetilde{\Mb}^{\ani}_G \ar[d] \\
        \widetilde{\A}^{G-\infty,\ani}_\CE \ar[r] & \widetilde{\A}^{\infty,\ani}        }
\end{equation*}

\end{construction}

\subsection{Description of the twisted inertia stack}
We continue with the situation of \cite[Section 5]{GWZ20b} and give the following description of the inertia stack of $I_{\hat\mu}\widetilde{\mathbb{M}}^{\mathrm{ani}}_{G, \bar k}$ which corrects and strengthens \cite[Theorems 5.14 and Corollary 5.24]{GWZ20b}. 

For this, we call two coendoscopy data $(\kappa,\rho_\kappa,\rho_\kappa \to \rho)$ and $(\kappa',\rho_{\kappa'},\rho_{\kappa'} \to \rho)$ for $G$ over $k$ equipped with a trivialization of $\rho_{\kappa,\infty}$ isomorphic, if $\kappa=\kappa'$ and if there exists an isomorphism $\rho_{\kappa}\cong \rho_{\kappa'}$ of $\pi_0(\kappa)$-torsors over $X$ compatible with the trivializations at $\infty$ and the morphisms to $\rho$.

\begin{theorem}\label{dec-IM}
	The morphism of stacks
	\begin{equation*}
		\bigsqcup_{\mathcal{E}}\tilde \mu_{\mathcal{E}} \colon  \bigsqcup_{\mathcal{E}} \widetilde{\mathbb{M}}^{G-\infty,\mathrm{ani}}_{\mathcal{E}} \to I_{\hat\mu}\widetilde{\mathbb{M}}^{\mathrm{ani}}_{G,\bar k},
	\end{equation*}
	where the disjoint union is taken over isomorphism classes of coendoscopic data $\mathcal{E}=(\kappa, \rho_\kappa, \rho_\kappa\to \rho_{\bar k})$ for $G$ over $\bar k$ equipped with a trivialization at $\infty$, is an equivalence.
\end{theorem}

\begin{rmk}\label{desc}
    One could phrase Theorem \ref{dec-IM} as an equivalence of stacks over $k$ as follows: The Galois group $\Gal(\bar k / k)$ acts naturally on the set of isomorphism classes of coendoscopic data $\CE$ appearing there, and for each orbit of this action, the stack $\bigsqcup_{\mathcal{E}} \widetilde{\mathbb{M}}^{G-\infty,\mathrm{ani}}_{\mathcal{E}}$ descends to $k$. Then the disjoint union of these stacks over $k$ is naturally equivalent to $I_{\hat\mu}\widetilde{\mathbb{M}}^{\mathrm{ani}}_{G}$ over $k$.
\end{rmk}

\begin{construction}[{\cite[Construction 4.36]{GWZ20b}}]\label{def-4.36}
	Let $S$ be a connected $X$-scheme equipped with $(E,\theta):S\to [\mathfrak g_D/G]$. 
	In applications $S$ lies over $\infty\in X$.  
	Assume the image of $(E,\theta)$ in $\mathfrak c_D$ is contained in $\mathfrak c_D^{\mathrm{rs}}$ and coincides with the image of $t\in \mathfrak t^{\mathrm{rs}}(S)$. Let us construct a map $\Aut(E,\theta)\to T(S)$. 
	
	Passing to étale covers we assume $E$ is trivial and there exists $b\in G(S)$ such that $\ad_b(\theta)=t$. Then $\gamma\in \Aut(E,\theta)$ can be identified with the right multiplication by $\gamma'\in I_\theta(S)$. 
	We use $\ad_b$ to identify $I_\theta$ with $I_t= T$. 
	We hence get an element $\gamma''\in T(S)$. 
	One can check that the resulting map $\Aut(E,\theta)\to T(S)$ doesn't depend on the choice of the trivialization and $b$. 
\end{construction}

\begin{proof}[Proof of Theorem \ref{dec-IM}]

	Let $S$ be a $\bar k$-scheme and $(E,\theta, \tilde{\infty})\in \tilde{\mathbb M}_{G}^{\mathrm{ani}}(S)$ equipped with $\gamma: \hat{\mu}\to \Aut (E,\theta)$. 
	Here $E$ is a $(\mathbb{G}\rtimes \mathrm{Out}(\mathbb G))$-torsor on $X\times S$, 
	$\theta$ is an equivariant map $E\to \mathbf{g}_{D} \times \rho_G$, and $\tilde{\infty}$ is a map $S\to \mathbf t^{\mathrm{rs}}$ lifting the characteristic $S\to \mathfrak c_\infty\cong  \mathbf c$ of $\theta_\infty$. 
	By restricting $\gamma$ to $\infty\times_k S$ and appling Construction \ref{def-4.36} we get a morphism $\gamma_\infty: \hat{\mu}\to {\mathbb T}$ on $S$. This morphism factors through a finite quotient $\mu_N\to {\mathbb T}$ and homomorphisms between finite type diagonalizable groups are Zariski locally constant by \cite[VIII, 1.5]{SGA3II}. Hence $\gamma_\infty$ arises by base change from a homomorphism $\kappa\colon \hat \mu \to \Tb$ over $k$.
	
	Now let $F \subset E$ be the subsheaf consisting of those sections of $E$ on which the actions of $\hat\mu$ through $\gamma$ and $\kappa$ coincide. By \cite[Lemma 5.11]{GWZ20b} this is a $(\mathbb{G} \rtimes \Out(\mathbb{G}))_\kappa$-subtorsor of $E$. Let $\rho_\kappa$ be the $\pi_0(\kappa)$-torsor $\pi_0(F)$ over $X_{\bar k} \times S$.  
	
	\begin{lemma}\label{trival-infty}
		The $\pi_0(\kappa)_\infty$-torsor $\rho_{\kappa,\infty}$ contains a canonical $S$-point. 
	\end{lemma}
	
	\begin{rmk}
		This statement does not appear in the original paper \cite{GWZ20b}. As a result of Lemma \ref{trival-infty}, the additional assumption \cite[Situation 6.2]{GWZ20b} holds automatically.

	\end{rmk}
	
	\begin{proof}[Proof of Lemma \ref{trival-infty}]
		Consider the fiber $K=\theta^{-1}(\tilde{\infty}, \infty_{\rho})\subseteq E_\infty$. 
		Since the centralizer of $\tilde{\infty}$ in $\mathbb G$ is precisely $\mathbb T$, the space $K$ is a $\mathbb T$-torsor over $S$. In particular $K$ has connected fibers over $S$. By Definition \ref{def-4.36} 
		the actions of $\hat\mu$ through $\gamma$ and $\kappa$ on $K$ coincide, so $K\subseteq F$. Thus $[K/\BT]$ defines a $S$-point in $\rho_{\kappa,\infty}=\pi_0(F)_\infty$. 
	\end{proof}
	Let $S'$ be a connected $\bar k$-scheme of finite type equipped with a $\bar k$-point $s$. By \cite[X. 1.7]{SGA1} we have an isomorphism 
	\[\pi_1(X_{\bar k}\times_{\bar k} S', \bar{\infty}\times s)\cong \pi_1(X_{\bar k},\bar{\infty})\times \pi_1(S',s).\] Using this and Lemma \ref{trival-infty} it follows that $\rho_\kappa$ descends canonically to a $\pi_0(\kappa)$-torsor over $X_{\bar k}$ which we again denote by $\rho_\kappa$. 
    
	We have a $(\mathbb{G} \rtimes \Out(\mathbb{G}))_\kappa$-equivariant morphism $F \to \rho_{\kappa,X_{\bar k} \times S}$ which factors through a 
    morphism $\rho_{\kappa,X_{\bar k} \times S} \to \rho_{X_{\bar k} \times S}$. This morphism is again already defined over $X_{\bar k}$.
	
    Since the image of $\gamma$ fixes $\theta$, it follows that the composition
	\begin{equation*}
		 \bar{\theta}_F: F \to E \to \mathbf{g}_D\times \rho \to \mathbf{g}_D
	\end{equation*}
	factors through the fixed points of $\kappa$ on $\mathbf{g}_D$, that is, through $\mathbf{h}_D$. Thus we obtain a $(\mathbb{G} \rtimes \Out(\mathbb{G}))_\kappa$-equivariant morphism $\theta_F: F \to \mathbf{h}_D \times \rho_\kappa$.

	We let $\mathcal{E}=(\kappa, \rho_\kappa, \rho_\kappa\to \rho)$ be the resulting coendoscopy datum for $G$ over $\bar k$ with a trivialization at $\infty$ given by Lemma \ref{trival-infty}.
	For any $x\in K(\bar k)$, by definition $\theta_F(x)=(\tilde{\infty}, \infty_{\rho_G})$. We conclude that the pair $((F, \theta_F), \tilde{\infty})$ defines an element in $\widetilde{\mathbb{M}}^{G-\infty,\mathrm{ani}}_{\mathcal{E}}(k)$. 
	One can check that the map $(E,\theta, \tilde{\infty})\mapsto ((F, \theta_F), \tilde{\infty})$ is an inverse of $\bigsqcup_{\mathcal{E}}\tilde \mu_{\mathcal{E}}$, which proves the theorem. 
\end{proof}

By taking $k$-groupoids in Theorem \ref{dec-IM}, we obtain the following statement:
\begin{corollary}[correction of {\cite[Theorem 5.14]{GWZ20b}}]
	The morphism of groupoids 
	\begin{equation*}
		\bigsqcup_{\mathcal{E}}\tilde \mu_{\mathcal{E}} \colon  \bigsqcup_{\mathcal{E}} \widetilde{\mathbb{M}}^{G-\infty,\mathrm{ani}}_{\mathcal{E}}(k) \to I_{\hat\mu}\widetilde{\mathbb{M}}^{\mathrm{ani}}_G(k),
	\end{equation*}
	where the disjoint union is taken over isomorphism classes of coendoscopic data $\mathcal{E}=(\kappa, \rho_\kappa, \rho_\kappa\to \rho)$ for $G$ over $k$ equipped with a trivialization at $\infty$, is an equivalence. 
\end{corollary}
\begin{proof}
    This follows from Theorem \ref{dec-IM} by noting that e.g. the proof of Theorem \ref{dec-IM} shows that for any point of $I_{\hat\mu}\widetilde{\mathbb{M}}^{\mathrm{ani}}_G(k)$, the associated coendoscopic datum $\CE$ is defined over $k$.
\end{proof}

We also need a version of this statement for rigidified stacks:
\begin{construction} \label{desc-IM-cor1}
       Let $\CE$ be a coendoscopic datum for $G$ over $k$ equipped with a trivialization at $\infty$. Since the actions of $\Out(\BG)$ on $\BG$ and $\Out(\BH)$ on $\BH_\kappa$ preserve $Z(\BG)$ and $Z(\BH_\kappa)$ we have $Z(G)=Z(\BG) \times^{\pi_0(\kappa)} \rho_\kappa$ and $Z(H_\CE)=Z(\BH_\kappa)\times ^{\pi_0(\kappa)} \rho_\kappa$. Furthermore, the inclusion $Z(\BG) \into Z(\BH_\kappa)$ is equivariant with respect to the action of $\pi_0(\kappa)$ on $Z(\BG)$ through $o_\BG$ and on $Z(\BH_\kappa)$ through $o_{\BH}$. Hence by twisting this inclusion with $\rho_\kappa$ we obtain a natural inclusion $Z(G) \into Z(H_\CE)$ over $X$. This induces an inclusion $Z(X,G) \into Z(X,H_\CE)$ of the Weil restrictions of these group schemes to $\Spec(k)$.

One sees e.g. from Proposition \ref{TwistedDesc} that $Z(G)=Z(\Gb) \times^{\pi_0(\kappa)} \rho_k$ naturally acts on every section of $[\mathfrak{h}_{ \CE,D}/H_\CE]^{\beta_\CE}$. Hence the group scheme $Z(X,H_\CE)$ acts naturally on every section of $\widetilde{\BM}^{G-\infty,\ani}_{H_\CE}$. So through the above inclusion $Z(X,G) \into Z(X,H_\CE)$ the same is true for the group $Z(X,G)$. We let $\widetilde{\Mc}^{G-\infty,\mathrm{ani}}_{\mathcal{E}}$ be the rigidification of $\widetilde{\mathbb{M}}^{G-\infty,\mathrm{ani}}_{\mathcal{E}}$ by this action of $Z(G,X)$.
\end{construction}

Then analogously to Corollary \ref{desc-IM-cor1}, by rigidifying the equivalence of Theorem \ref{dec-IM} and taking $k$-groupoids we obtain:
\begin{corollary}[correction of {\cite[Corollary 5.24]{GWZ20b}}]\label{dec-IM2}
		The morphism of groupoids 
	\begin{equation*}
		\bigsqcup_{\mathcal{E}}\tilde \mu_{\mathcal{E}} \colon  \bigsqcup_{\mathcal{E}} \widetilde{\Mc}^{G-\infty,\mathrm{ani}}_{\mathcal{E}}(k) \to I_{\hat\mu}\widetilde{\Mc}^{\mathrm{ani}}_G(k),
	\end{equation*}
	where the disjoint union is taken over isomorphism classes of coendoscopic data $\mathcal{E}=(\kappa, \rho_\kappa, \rho_\kappa\to \rho)$ for $G$ over $k$ equipped with a trivialization at $\infty$, is an equivalence. 
\end{corollary}

\section{Comparing stabilised point-counts}\label{stpc}
To substitute the moduli spaces of twisted Higgs bundles into the main argument of \cite{GWZ20b}, we need the following result comparing stabilized points counts for moduli spaces of twisted and untwisted Higgs bundles. For this we work over a finite field $k$ and fix a smooth connected projective curve $X$ over $k$ as well as a quasi-split reductive group scheme $H$ over $X$.

Throughout this section we work over the anisotropic locus $\widetilde{\A}_{H}^{ani} \subset \widetilde{\A}_{H}$. To alleviate the notation, we will often drop the superscript $\ani$ from now on, for example we will write $\widetilde{\Mc}_{\Ee}$ instead of $\widetilde{\Mc}_{\Ee}^{ani}$ etc. By Proposition \ref{heart}, $\widetilde{\Mc}_{H}$ and $\widetilde{\Mc}_{H}^\beta$ are smooth DM-stacks, proper over $\widetilde{\A}_{H}$.

\begin{theorem}\label{stabct}
For any $Z$-gerbe $\beta$ over $X$ and any $a \in \widetilde{\A}_{H}(k)$ the following identity holds:
\[\#^{stab} \widetilde{\Mc}^\beta_{H,a}(k) = \#^{stab} \widetilde{\Mc}_{{H},a}(k).\]
\end{theorem}

\begin{rmk}
    As observed in \cite[Remark 1.2]{GWZ20a}, the fact that $ \mathbb{M}^\beta_{H,a}\to\mathcal{M}^\beta_{H,a}$ is a gerbe banded by $Z$ implies they have isomorphic $\ell$-adic cohomology groups, so
	\[\#^{stab} \widetilde{\mathbb{M}}^\beta_{H,a}(k)=\#^{stab} \widetilde{\mathcal{M}}^\beta_{H,a}(k),\qquad
    \#^{stab} \widetilde{\mathbb{M}}_{H,a}(k)=\#^{stab} \widetilde{\mathcal{M}}_{H,a}(k).\]
\end{rmk}
\begin{corollary}\label{cor:stabcoh}
Under the same assumptions as in Theorem \ref{stabct}, for any prime $\ell$ different from the characteristic of $k$, there exists an abstract isomorphism of stabilised $\ell$-adic cohomology groups
$$H^*(\widetilde{\mathcal{M}}^\beta_{H,a})^{stab} \simeq H^*(\widetilde{\mathcal{M}}_{H,a})^{stab}.$$
\end{corollary}
\begin{proof}
By smoothness of the Deligne-Mumford stack $\widetilde{\mathbb{M}}^\beta_{H}$ and properness of the Hitchin map $\chi^{\beta}_H$, the derived pushforward $R(\chi_H^{\beta})_* \overline{\Qb}_{\ell}$ is pure by virtue of Deligne's Weil II, \cite{deligne1980conjecture}. In particular, the direct summand of $\pi_0(\Pc_H)$-invariants is also pure.

This property holds for all possible central twists, including the trivial one. The comparison of point-counts proven in Theorem \ref{stabct} yields therefore an equality of Frobenius traces of the complexes of sheaves 
$(R(\chi^{\beta}_H)_* \overline{\Qb}_{\ell})^{stab}$ and $(R(\chi_H)_* \overline{\Qb}_{\ell})^{stab}$. By Chebotarev Density, their semisimplifications are isomorphic. Since these objects in the derived category of constructible $\ell$-adic sheaves are already semisimple by virtue of the Decomposition Theorem \cite{BBD82}, we conclude the existence of an isomorphism of $(R(\chi^{\beta}_H)_* \overline{\Qb}_{\ell})^{stab}$ and $(R(\chi_H)_* \overline{\Qb}_{\ell})^{stab}$. Passing to stalks, and applying Proper Base Change we obtain the requisite statement.
\end{proof}

In the following two subsections, we give two different proofs of this fact, one using $p$-adic integration and one using a product formula for $\Mb^\beta_{H}$.

\subsection{Proof via $p$-adic integration}

We write $F=k((t))$ and $\Oc = k[[t]]$. In this section we consider the pullback of $ \widetilde{\Mc}_{H}^\beta$ and $\widetilde{\Mc}_{{H}}$  to the formal disc $\Spec(\Oc)$, or equivalently we could pull back the curve $X$ and all the auxiliary data to $\Spec(\Oc)$ and consider the relative moduli stacks. 

The regular loci $\Mc_{H}^{\beta,\reg}$  and $\Mc_{H}^{\reg}$ are both $\Pc_{H}$-torsors by Proposition \ref{heart} and \cite[Proposition 4.3.3]{Ng10} respectively. In particular there exists a $\Pc_{H}$-torsor $\tau$ over $\widetilde{\A}_{H}$ such that 

\begin{equation}\label{tauiso}
\Mc_{H}^{\beta,\reg} \cong \Mc_{H}^{\reg,\tau} =\Mc_{H}^{reg} \times^{\Pc_{H}}_{\widetilde{\A}_{H}} \tau. \end{equation}

\begin{rmk}
    The equivalence \eqref{tauiso} does not extend beyond the regular locus, as can already be seen in the situation of Example \ref{slex}. The moduli stack $\Mc^\beta_{\SL_n}$ is smooth if $\beta$ comes from a line bundle with degree coprime to $n$, but singular in general. 
\end{rmk}

Since the coarse moduli spaces are proper over the Hitchin base, \eqref{tauiso} induces an isomorphism of $F$-analytic manifolds
\begin{equation}\label{mfld}\widetilde{M}^\beta_{{H}}(\Oc)^\natural \cong \widetilde{M}_{H}^{\tau}(\Oc)^\natural,  \end{equation}
where $ \widetilde{M}^\tau_{H}(\Oc)^\natural =  \widetilde{M}^\tau_{H}(\Oc) \cap  \widetilde{M}_{H}^{reg,\tau}(F)$ and $\widetilde{M}^\beta_{H}(\Oc)^\natural = \widetilde{M}^\beta_{H}(\Oc) \cap \widetilde{M}_{H}^{\beta,reg}(F)$.

\begin{lemma}\label{lift} A point $x \in \widetilde{M}^\beta_{H}(\Oc)^\natural $ lifts to $\widetilde{\Mc}^\beta_{H}(\Oc)$
if and only if the image of $x$ under \eqref{mfld} lifts to $\widetilde{\Mc}^\tau_{H}(\Oc)$.   
\end{lemma}
\begin{proof} Let us assume first $x$ that lifts to $\widetilde{\Mc}^\beta_{H}(\Oc)$. By general properties of tame DM-stacks,\cite[Proposition 2.12 and Construction 2.13]{GWZ20a}, the image of $x$ under \eqref{mfld} lifts to an $[\Oc_L/\mu_r]$-point of $\widetilde{\Mc}^\tau_{H}$, where $\Oc_L$ denotes the ring of integers of a totally ramified extension $L$ of $F$ of degree $r$. In order to show that $r=1$ we may replace $k$ by $\overline{k}$ and $\Oc$ by $\Oc^{ur} = \overline{k}[[t]]$. Then by definition the Z-gerbe $\beta$ is torsion of some order $N$. Thus by \cite[\href{https://stacks.math.columbia.edu/tag/0AMB}{Tag 0AMB}]{stacks-project} we may pass to a degree $N$ cover  $Y  \to X$ such that the pull back of $\beta$, and thus also of $\tau$, is trivial over $Y$. Thus the moduli stacks $\widetilde{\Mc}^\beta_{H}(Y)$ and $ \widetilde{\Mc}_{H}(Y)$ are isomorphic separated DM-stacks and therefore an $F$-point extends to an $\Oc$-point in one if and only if it does in the other. The converse implication is proven the same way. 
\end{proof}

In order to compare integrals on $\widetilde{M}^\beta_{H}(\Oc)^\natural \cong \widetilde{M}_{H}^{\tau}(\Oc)^\natural$ we use that by Proposition \ref{heart} both spaces are torsors for a smooth group scheme on the regular locus, whose complement has codimension $2$. Thus picking a relative trivializing form we obtain volume forms $\omega_{\beta}$ and $\omega_{H}$ on both spaces as in \cite[Lemma 6.13]{GWZ20b}, compatible with \eqref{tauiso}. 
Finally, let us denote by $\widetilde{M}^\beta_{H}(\Oc)^{lift} \subset \widetilde{M}^\beta_{H}(\Oc)^\natural$ the subset of $\Oc$-point that lift to $\widetilde{\Mc}^\beta_{H}(\Oc)$ and similarly $ \widetilde{M}_{H}^{\tau}(\Oc)^{lift}\subset  \widetilde{M}_{H}^{\tau}(\Oc)^\natural$. Further for $a \in \widetilde{\A}_{H}(k)$ we write $\widetilde{M}^\beta_{H}(\Oc)^{lift}_a,\widetilde{M}_{H}(\Oc)^{lift}_a$ for the subsets $\Oc$-points which specialize to $a$ in the Hitchin base. 

Then the orbifold formula \cite[Theorem 2.21]{GWZ20b} together with Lemma \ref{lift} implies for every $a \in \widetilde{\A}_{H}(k)$

\[ \#\widetilde{\Mc}^\beta_{H,a}(k) =q^{\dim \widetilde{\Mc}^\beta_{H}} \int_{\widetilde{M}^\beta_{H}(\Oc)^{lift}_a }|\omega_{\beta}| =q^{\dim \widetilde{\Mc}^\tau_{H}}\int_{\widetilde{M}_{H}(\Oc)^{lift}_a} |\omega_{H}| =\#\widetilde{\Mc}^{\tau}_{H,a}(k),\]

In order to be able to compare the stable point counts we repeat the above argument for a given $a \in \widetilde{\A}_{H}(k)$ with the unramified twists $\widetilde{\Mc}_{H,a}^{\beta,t}$ and $\widetilde{\Mc}_{H,a}^{\tau+t}$ as in \cite[Section 6.2]{GWZ20b}, for $t \in  H^1_{\text{\'et}}(k,\pi_0(\Pc_{H,a}))\cong \pi_0(\Pc_{H,a})(k) $, to obtain
\[\#\widetilde{\Mc}_{H,a}^{\beta,t}(k) = \#\widetilde{\Mc}_{H,a}^{t+\tau}(k).\]
By \cite[Lemma 6.6]{GWZ20b} we then finally obtain

\begin{align*}\#^{stab} \widetilde{\Mc}^\beta_{H,a}(k) &=\frac{1}{|\pi_0(\Pc_{H,a})(k)|} \sum_{t \in \pi_0(\Pc_{H,a})(k)} \#\widetilde{\Mc}_{H,a}^{\beta,t}(k)\\ &=\frac{1}{|\pi_0(\Pc_{H,a})(k)|} \sum_{t \in \pi_0(\Pc_{H,a})(k)} \#\widetilde{\Mc}_{H,a}^{t+\tau}(k) =   \#^{stab} \widetilde{\Mc}_{H_\kappa,a}(k).\end{align*}

\subsection{Proof via an analogue of Ngô's product formula}\label{sect:product formula}

\begin{definition}
    Let $J/(\mathfrak{c}_h)_D$ be the regular centralizer. Fix $a\in \mathsf{A}_H^{\ani}(k)$.
    \begin{enumerate}
        \item There is a natural morphism $Z(H)\times_X (\mathfrak{c}_h)_D\to J$ of group schemes over $(\mathfrak{c}_h)_D$ that is compatible with their embedding into $G$. 
        Let $\beta_J$ be the $J$-gerbe $B_{(\mathfrak{c}_h)_D}J\times^{B_{(\mathfrak{c}_h)_D}Z(H)} \beta$ on $(\mathfrak{c}_h)_D$. 
        We can identify $[h_{D}/H]^\beta$ with $[h_{D}/H]\times^{B_{(\mathfrak{c}_h)_D}J} \beta_J$. The $B_{(\mathfrak{c}_h)_D}J$-action on $[h_{D}/H]$ defines a $B_{(\mathfrak{c}_h)_D}J$-action on $[h_{D}/H]^\beta$. 

        \item Let $J_a$ be the commutative group scheme $a^*J$ on $X$, and let $h_{D,a}=\chi^{-1}(a)\subset h_D$. 
        The fiber of $\mathbb{M}_{H}^\beta\to \mathsf{A}_H$ at $a$, denoted by $\mathbb{M}_{H,a}^\beta$, is identified with the moduli stack of sections of $[h_{D,a}/H]^\beta\to X$. The fiber of the Prym stack $\mathbb{P}_{H}\to \mathsf{A}_H$ at $a$, denoted $\mathbb{P}_{H,a}$, is identified with the moduli stack of $J_a$-torsors on $X$.
        
        \item Let $\beta_a$ be the $J_a$-gerbe on $X$ given by $a^*\beta_J\cong B_XJ_a\times^{B_XZ(H)} \beta$. 
        Then $[h_{D,a}/H]^\beta\cong [h_{D,a}/H]\times^{B_{X}J_a} \beta_a$.

        \item 
        For any closed point $\nu \in X$ we denote by $X_\nu$ the formal disc around $\nu$ and let $\mathbb{P}_{H, a,\nu}$ be the Picard stack of $J_{a,\nu}$-torsors on the formal disk $X_\nu$ (denoted $\mathcal{P}_{a,\nu}$ in \cite[p. 20]{MR2218781}). Let $\mathbb{M}_{H,a,\nu}^\beta$ be the twisted variant of the stack $\mathcal{M}_{a,\nu}$ introduced in \cite[p. 20]{MR2218781}. More precisely, $\mathbb{M}_{H,a,\nu}^\beta$ associates a $k$-scheme $S$ to the groupoid of maps $X_\nu\hat{\times}_k S\to [h_D/H]^\beta$. The $B_XJ_{a,\nu}$-action on $[h_{D}/H]^\beta$ induces a $\mathbb{P}_{H, a,\nu}$-action on $\mathbb{M}_{H,a,\nu}^\beta$. 
    \end{enumerate}
\end{definition}

\begin{proposition}\label{prop:productformula}
    Let $a\in \mathsf{A}_H^{\ani}(k)$, $U=a^{-1}(\mathfrak{c}_D^{\mathrm{rs}})$. 
    Then we have a $\Gal(\bar k/k)$-equivariant equivalence of categories
		\[[\mathbb{M}_{H,a}^\beta/\mathbb{P}_{H,a}](\bar k)\cong \prod_{\nu\in X-U}[\mathbb{M}_{H,a,\nu}^\beta/\mathbb{P}_{H, a,\nu}](\bar k),\]
	In particular, we have an equivalence between categories of $k$-points
		\[[\mathbb{M}_{H,a}^\beta/\mathbb{P}_{H,a}](k)\cong
		\prod_{\nu\in |X-U|}[\mathbb{M}_{H,a,\nu}^\beta/\mathbb{P}_{H, a,\nu}](k).\]
\end{proposition}

\begin{proof}
    The restriction defines a $k$-morphism
		\[\gamma: [\mathbb{M}_{H,a}^\beta/\mathbb{P}_{H,a}]\to \prod_{\nu\in |X-U|}[\mathbb{M}_{H,a,\nu}^\beta/\mathbb{P}_{H, a,\nu}].\]
		It remains to check that this map induces an isomorphism on $\bar{k}$-points. Note that $J_a$ is a torus on the generic fiber. By Lemma \ref{torus-curve-trivial}, the group $H^2_{\et}(X_{\bar k},J_a)$ is trivial. In particular, the gerbe $\beta_{a,\bar k}$ is trivial. After a base change of $k$ to a finite extension, we may assume $\beta_a$ is trivial, so we may identify $[h_{D,a}/H]^\beta$ with $[h_{D,a}/H]$. 
        Then the claim follows from \cite[4.6]{MR2218781}.
\end{proof}

\begin{proposition}\label{comp-stab-prod}
    For any  $a\in \tilde{\mathsf{A}}^{\ani}_H(k)$, we have an isomorphism between $[\mathbb{M}_{H,a}^\beta/\mathbb{P}_{H,a}](k)$ and $[\mathbb{M}_{H,a}/\mathbb{P}_{H,a}](k)$. 
\end{proposition}

\begin{proof}
    For any $\nu\in |X-U|$,  the fact that $k(\nu)$ is a finite field implies that $J_a(\overline{k(\nu)})$ is a torsion group. Consequently, the group $H^2_{\et}(\nu, J_a)\cong H^2(\hat{\BZ}, J_a(\overline{k(\nu)}))$ is trivial. It follows that $\beta_a(\nu)$ is non-empty. By \cite[2.1.4]{Bouthier2022}, the pullback map $\beta_a(X_\nu)\to \beta_a(\nu)$ is essentially surjective. In particular, $\beta_a(X_\nu)$ is non-empty, so the $J_{a,\nu}$-gerbe $\beta_a|_{X_\nu}$ is neutral. Consequently, there is an isomorphism $[\mathbb{M}_{H,a,\nu}^\beta/\mathbb{P}_{H, a,\nu}]\cong [\mathbb{M}_{H,a,\nu}/\mathbb{P}_{H, a,\nu}]$. 
    The claim then follows from Proposition \ref{prop:productformula}. 
\end{proof}

By a variant of Grothendieck–Lefschetz trace formula \cite[11.1.7]{wang2024multiplicative} we have
    \[\#^{stab}\mathbb{M}_{H,a}^\beta(k)=\sum_n (-1)^n \Tr(\mathrm{Frob}, H^n_{\et,c}(\mathbb{M}_{H,a}^\beta,\overline{\mathbb{Q}}_\ell)^1)=\# \mathbb{P}_{H,a}^0(k)\# [\mathbb{M}_{H,a}^\beta/\mathbb{P}_{H,a}](k),\]
    \[\#^{stab}\mathbb{M}_{H,a}(k)=\sum_n (-1)^n \Tr(\mathrm{Frob}, H^n_{\et,c}(\mathbb{M}_{H,a},\overline{\mathbb{Q}}_\ell)^1)=\# \mathbb{P}_{H,a}^0(k)\# [\mathbb{M}_{H,a}/\mathbb{P}_{H,a}](k).\]
    Here $H^n_{\et,c}(\mathbb{M}_{H,a},\overline{\mathbb{Q}}_\ell)^1$ denotes the $\pi_0(\mathbb{P}_{H,a})$-invariant subspace of $H^n_{\et,c}(\mathbb{M}_{H,a},\overline{\mathbb{Q}}_\ell)$, and $\mathbb{P}_{H,a}^0$ denotes the neutral component of $\mathbb{P}_{H,a}$. 
    We conclude that \[\#^{stab}\mathbb{M}_{H,a}^\beta(k)=\#^{stab}\mathbb{M}_{H,a}(k).\]

\begin{rmk}
    Let $X_\nu^\bullet\subset X_\nu$ be the punctured formal disk. 
    Once we have chosen $E^*\in \mathbb{M}_{H,a,\nu}^{\beta}(k)$, as in \cite[p. 21]{MR2218781}, we can define the $\beta$-twisted Springer fibre  $\mathbb{M}_{H,a,\nu}^{\beta,\bullet}$. It represents the functor that assigns to a $k$-scheme $S$ the groupoid of pairs $(E_\nu,\iota_\nu)$, where $E_\nu:  X_\nu\hat{\times}_k S\to [h_D/H]^\beta$ and $\iota_\nu$ is an isomorphism $E_\nu|_{X_\nu^\bullet\hat{\times}_k S}\to E^*|_{X_\nu^\bullet\hat{\times}_k S}$. 
    Let $\mathbb{P}_{H,a,\nu}^{\bullet}$ be the Picard scheme of $J_{a,\nu}$-torsors on the formal disk $X_\nu$ with a trivialization over  $X_\nu^\bullet$.
    Due to the triviality of $H^2_{\et}(X_\nu, J_a)$ as explained in the proof of Proposition \ref{comp-stab-prod}, $\mathbb{M}_{H,a,\nu}^{\beta,\bullet}$ is (non-canonically) isomorphic to the usual Springer fibre $\mathbb{M}_{H,a,\nu}^{\bullet}$.
    
    As a corollary, $\mathbb{M}_{H,a,\nu}^{\beta,\bullet}$ is represented by an ind-scheme, equipped with an action by the group ind-scheme $\mathbb{P}_{H,a,\nu}^{\bullet}$. Moreover, \cite[4.6]{MR2218781} implies that the forgetful functor yields an isomorphism $$[\mathbb{M}_{H,a,\nu}^{\beta,\bullet}/\mathbb{P}_{H,a,\nu}^{\bullet}]\cong [\mathbb{M}_{H,a,\nu}^{\beta}/\mathbb{P}_{H,a,\nu}].$$
\end{rmk}

\bibliographystyle{amsalpha}
\bibliography{master}
\end{document}